\documentclass[12pt, reqno]{amsart}
\usepackage{mathrsfs}
\usepackage{amsfonts}
\usepackage[centertags]{amsmath}
\usepackage{amssymb,array}
\usepackage{amsthm}
\usepackage{stmaryrd}
\usepackage{graphicx}
\usepackage{caption}
\usepackage{ytableau}
\usepackage{enumerate,enumitem}
\usepackage{MnSymbol}
\SetLabelAlign{center}{\hfill #1\hfill}
\usepackage[textwidth=16cm, hmarginratio=1:1]{geometry}
\usepackage{tikz-cd, tikz}
\usepackage{rotating}
\usepackage[all,cmtip]{xy}
\usepackage[titletoc, title]{appendix}
\usepackage{amscd}
\usepackage[colorlinks]{hyperref}
\usepackage{float}
\usepackage{dsfont} 
\hypersetup{
bookmarksnumbered,
pdfstartview={FitH},
breaklinks=true,
linkcolor=blue,
urlcolor=blue,
citecolor=blue,
bookmarksdepth=2
}

\theoremstyle{plain}
   
   \newtheorem{theorem}{Theorem}[section]

   \newtheorem{lemma}[theorem]{Lemma}
   \newtheorem{corollary}[theorem]{Corollary}

\theoremstyle{definition}
   \newtheorem{definition}[theorem]{Definition}
   
   \newtheorem{example}[theorem]{Example}

   \newtheorem{remark}[theorem]{Remark}

\numberwithin{equation}{section}

\newcommand{\be}{\begin{enumerate}}

    \newcommand{\ene}{\end{enumerate}}

    \newcommand{\QQ}{\mathbb{Q}}
    
    \newcommand{\NN}{\mathbb{N}}
    \newcommand{\PP}{\mathbb{P}}
    
    \newcommand{\CC}{\mathbb{C}}

    \newcommand{\Hom}{\operatorname{Hom}}

    \newcommand{\Ima}{\operatorname{Im}}

    \newcommand{\Gr}{{\operatorname{Gr}}}
    \newcommand{\GL}{{\operatorname{GL}}}
    \newcommand{\undd}{\underline{\operatorname{dim}}}
    
    \newcommand{\odim}{\operatorname{dim}}

    \newcommand{\HH}{\operatorname{H}}

    \newcommand{\Irr}{\operatorname{Irr}}
    
    \newcommand{\fg}{\operatorname{\mathfrak{g}}}

    \newcommand{\fn}{\operatorname{\mathfrak{n}}}

    \newcommand{\la}{\langle}
    \newcommand{\ra}{\rangle}

    \newcommand{\bfi}{\mathbf{i}}
     \newcommand{\bfI}{\mathbf{I}}
    \newcommand{\bfj}{\mathbf{j}}

    \newcommand{\cF}{\mathcal{F}}

    \newcommand{\cB}{\mathcal{B}}

    \newcommand{\cE}{\mathcal{E}}
    
    \newcommand{\CM}[1]{\mathcal{M}}

    \newcommand{\adb}{{\alpha,\beta}}
    \newcommand{\apb}{{\alpha+\beta}}

    \newcommand{\tM}[1]{\widetilde{{#1}}}

    \newcommand{\squw}[2]{[{#1}_{1}{#1}_{2}\cdots{#1}_{#2}]}

    \newcommand{\Seteq}[2]{=\left\{{#1}\mid\,{#2}\right\}}

    \newcommand{\Ind}{\operatorname{Ind}}

\newlength{\mysizetiny}
\newlength{\mysizesmall}
\newlength{\mysize}
\newlength{\mysizelarge}
\begin{document}

\title{Geometrizations of quantum groups and dual semicanonical bases.}
\author{Yingjin Bi}
\address{Department of Mathematics, Harbin Engineering University}
\email{yingjinbi@mail.bnu.edu.cn}
\date{} 

\begin{abstract}
In this paper, we give a geometrization of semicanonical bases of quantum groups via Grothendieck groups of the derived categories of Lusztig's nilpotent varieties. Additionally, we provide a characterization of the dual semicanonical bases using Serre polynomials of Grassmannians of modules over preprojective algebras.
\end{abstract}

\maketitle 

\section{Introduction}
The current study aims to extend the concept of semicanonical bases $\mathcal{S}$ in the enveloping algebra $U(\mathfrak{n})$ to the semicanonical bases $\mathcal{S}_q$ in the quantum groups $U_q(\mathfrak{n})$ such that $\mathcal{S}_{q=1}=\mathcal{S}$. In \cite{lusztig2000semicanonical}, Lusztig introduced the semicanonical bases of the enveloping algebra $U(\mathfrak{n})$ based on Lusztig's nilpotent varieties. Following Lusztig's method, Geiss, Leclerc, and Schröer utilized modules over preprojective algebras to construct the dual semicanonical bases of the coordinate rings of unipotent groups $\mathbb{C}[N]$, see \cite{geiss2005semicanonical},\cite{geiss2007semicanonical}. 

It is essential to introduce dual semicanonical bases as they are connected to the cluster structure on the coordinate rings $\mathbb{C}[N]$ of unipotent subgroups $N$ through the representations of preprojective algebras, see \cite{geiss2005semicanonical}. For quantum cluster algebras, in \cite{geiss2013cluster}, Geiss, Leclerc, and Schröer demonstrated a quantum cluster algebra structure on the quantum coordinate rings $A_q(\mathfrak{n}(w))$ for a specific Weyl group element $w \in W$. All cluster monomials in $A_q(\mathfrak{n}(w))$ are part of the dual semicanonical bases. The conjecture that these cluster monomials also belong to the dual canonical bases was proven by Qin for the Dynkin case and all symmetrizable Kac-Moody types, see \cite{qin2017triangular} and \cite{qin2020dual}. In \cite{kang2018monoidal}, Kang, Kim, Kashiwara, and Oh utilized $R$-matrices of modules over quiver Hecke algebras $R_Q$ to validate this conjecture for symmetric Kac-Moody types. This conjecture establishes a relationship between dual canonical bases and dual semicanonical bases. 

However, there is limited research on the quantum counterpart of dual semicanonical bases. One key question is how to characterize these dual semicanonical bases. Our primary result is describing dual semicanonical bases in terms of Serre polynomials of Grassmannians of modules over preprojective algebras. Building on Leclerc's findings, we can depict the dual canonical bases using shuffle algebras. Similarly, we will present the explicit form of dual semicanonical bases using shuffle algebras.\\

\noindent
Let us interpret our main results. For a dimension vector $\alpha\in\NN[Q_0]$ of a quiver $Q$, where $Q_0$ is the set of vertices of $Q$. Let $\Lambda_\alpha$ denote the Lusztig's nilpotent variety of dimension vector $\alpha$, and $G_\alpha$ refer to the algebraic group $\prod_{i\in Q_0}\GL(\alpha_i)$. Define $D_{G_\alpha}^b(\Lambda_\alpha)$ as the $G_\alpha$-equivariant bounded derived category of constructible sheave on $\Lambda_\alpha$, and as  $K(\Lambda_\alpha)$ its Grothendieck group.

For a partition of the dimension vector $\alpha=\beta+\gamma$, introduce an Induction functor from $D_{G_\beta}^b(\Lambda_\beta)\times D_{G_\gamma}^b(\Lambda_\gamma)$ to $D_{G_\alpha}^b(\Lambda_\alpha)$. This functor leads to a product on $K(\Lambda_\beta)\times K(\Lambda_\gamma)$. Let $K^{nil}(\Lambda_Q)$ be the subalgebra of $K(\Lambda_Q)$ generated by the constant sheaves $\CC_i$ of $\Lambda_{\alpha_i}$ for all $i\in Q_0$, where $\alpha_i$ represents the simple root for $i$. We say two complexes $\cF$ and $\cE$ in $D_{G_\alpha}^b(\Lambda_\alpha)$ \emph{almost equal} if there is an open dense subset $U_\alpha\subset \Lambda_\alpha$ such that $\cF_{\mid U_\alpha}=\cE_{\mid U_\alpha}$. let $\tM{K}^{nil}(\Lambda_\alpha)$ be the quotient of $K^{nil}(\Lambda_\alpha)$ by the almost equal relation. The dual space of $\tM{K}^{nil}(\Lambda_Q)$ is denoted as $\tM{K}^*(\Lambda_Q)$ . The following theorem is one of our main results.
\begin{theorem}
    For a quiver $Q$ without multiple arrows between any two vertices, let $U_q(\fn_Q)$ be the positive-part of the quantum group corresponding to $Q$. There exists an $\CC(q)$-algebra isomorphism 
    \[\Phi:\, \tM{K}^{nil}(\Lambda_Q)\cong U_q(\fn_Q)\]
    such that the constant sheaf of irreducible components forms a basis of $U_q(\mathfrak{n}_Q)$, known as the \emph{semicanonical base}.
\end{theorem}
 For a dimension vector $\alpha\in\NN[Q_0]$, let $\la Q_0\ra_\alpha$ be the set of words $\bfi=[i_1i_2\cdots i_n]$ of $Q_0$ with $\sum_{k=1}^n i_k=\alpha$. Define $\CC(q)\la Q_0\ra_\alpha$ as the $\CC(q)$-vector space generated by all $\bfi\in\la Q_0\ra_\alpha$. Let $\CC(q)\la Q_0\ra$ denote the space $\bigoplus_{\alpha\in \NN[Q_0]}\CC(q)\la Q_0\ra_{\alpha}$. Following \cite{leclerc2004dual}, there exists an embedding from $A_q(\fn_Q)$ to $\CC(q)\la Q_0\ra$. 

\begin{theorem}
 For a quiver $Q$ without multiple arrows between any two vertices, we can embed the dual space $\tM{K}^*(\Lambda_Q)$ into $\CC(q)\la Q_0\ra$, where the image of any irreducible component $Z$ of $\Lambda_\alpha$ is equal to (up to some $q$-power)
 \[\delta_Z:=\sum_{\bfi\in\la Q_0\ra_\alpha} q^{-d_\bfi}\chi_q(\Gr_\bfi(z))\bfi \] 
 where $\chi_q(\Gr_\bfi(z))$ is the Serre polynomial of the Grassmannians $\Gr_\bfi(z)$ of a generic point $z\in Z$ for any words $\bfi\in\la Q_0\ra_{\alpha}$ and $d_\bfi=\odim\Lambda_{\bfi,\alpha}$, see (\ref{eq_lambfi}) and (\ref{eq_grassman}).
\end{theorem}

\section{Derived categories of Lusztig's nilpotent varieties}
In this section, we review some concepts related to derived categories of constructible sheaves on Lusztig's nilpotent varieties. For a variety $X$ over $\CC$, we say a sheaf $\cF$ \emph{constructible }if there is a decomposition of $X=\sum_i X_i$ such that $X_i$ is a constructible subset and $\cF_{\mid X_i}$ is constant for each $i$. For example, for a constructible subset $U\subset X$, the sheaf $\bfI_U$, defined by $\bfI_U(u)=1$ for $u\in U$ and $\bfI_U(v)=0$ for $v\notin U$, is a constructible sheaf. We write the $D^b(X)$ for the bounded derived category of constructible sheave on $X$. 

If there exists an action of an algebraic group $G$ on $X$, we refer to a constructible sheaf $\mathcal{F}$ as \emph{equivariant} if $\mathcal{F}_{\mid G\cdot x}$ is constant for any $x \in X$. The derived category of equivariant constructible sheaves over $X$ is denoted as $D_G^b(X)$. For a $G$-invariant morphism $f:X \to Y$, we can define $f_!,f_*: D_G^b(X) \to D_G^b(Y)$ and $f^*,f^! : D_G^b(Y) \to D_G^b(X)$ as detailed in \cite{bernstein2006equivariant}. \\

\noindent
The Grothendieck group of $D_G^b(X)$ is denoted as $K_G(X)$. It is a module over the ring $ \CC[q,q^{-1}]$ defined by $q\mathcal{F}=\mathcal{F}[1]$, where $[1]$ represents the shift of the complex $\mathcal{F} \in D_G^b(X)$. For a constant sheaf $\mathbf{I}_X$ and the canonical projection $p:X \to pt$, we have $p_!\mathbf{I}_X=\sum_i \HH_i(X)$ in $K_G(pt)$ where $\HH_\bullet(X)$ denotes the Borel-Moore homology of $X$.

\subsection{Constructible functions}
Consider an algebraic variety $X$ defined over the field of complex numbers. We refer to the space of all constructible functions $f:X\to \QQ$ as $M(X)$, which is a $\QQ$-vector space. This means that $f^{-1}(a)$ is constructible for every $a\in \QQ$, and is empty for all but finitely many $a\in \QQ$.

The map $\int_X: M(X)\to \QQ$ is a $\QQ$-linear map defined as $f\mapsto \int_X f=\sum_{a\in\QQ}a\chi(f^{-1}(a))$, where $\chi$ represents the Euler characteristic. If there exists an algebraic group $G$ that acts on $X$, we define $M_G(X)$ as the $\QQ$ vector space comprising functions $f\in M(X)$ such that $f$ remains constant on every $G$-orbit within $X$.

Next, we introduce a map from $K_G(X)$ to $M_G(X)$. For a complex $\cF\in D_G^b(X)$, we define a constructible function $\Psi(\cF)$ as $$\Psi(\cF)(x):=\sum_i(-1)^i\dim\HH^i(\cF)_x.$$ Since $\cF$ is a bounded complex, $\Psi(\cF)$ is well-defined. This function induces a map as follows.
\begin{equation}
\begin{split}
       \Psi: K_G(X)&\to M_G(X)\\ 
          \cF&\mapsto \Psi(\cF)
\end{split}
\end{equation}

\subsection{Lusztig's nilpotent varieties}
In this section, we revisit the concept of Lusztig's nilpotent varieties. Given a quiver $Q=(Q_0,Q_1)$, we define its double quiver $\overline{Q}$ by adding the opposing arrows $\bar{\alpha}$ for each $\alpha\in Q_1$. We assign a function $\varepsilon:\overline{Q}\to \{1,-1\}$ such that $\varepsilon(\alpha)=1$ if $\alpha\in Q_1$ and $\varepsilon(\bar{\alpha})=-1$ if $\bar{\alpha}\in Q_1^{op}$.

The \emph{preprojective algebra} $\Lambda_Q$ of $Q$ is defined as the quotient algebra $k\overline{Q}/I$ where $I$ is the ideal generated by the equation
\begin{equation}\label{eq_prepro}
    \sum \epsilon(\alpha) \alpha\bar{\alpha}=0.
\end{equation}

For a dimension vector $\alpha$, we define $\Lambda_\alpha$ as the variety of modules over $\Lambda_Q$ with dimension vector $\alpha$. For a given path $h=\alpha_1\cdots\alpha_n$, we associate it with an element $z_h=z_{\alpha_1}\cdots z_{\alpha_n}$ for any $z\in \Lambda_\alpha$. An element $z\in\Lambda_\alpha$ is deemed \emph{nilpotent} if $z_h=0$ for any path $h$ of sufficient length. We abbreviate $\Lambda_\alpha$ for the subvariety of $\Lambda_\alpha$ consisting of nilpotent elements, if there is no any confusion.

\subsection{Flag varieties}
Let us recall the flag varieties. For a $Q_0$-graded vector space $V$ with dimension vector $\gamma$ and a word $\bfi=[i_1\cdots i_n]\in\la Q_0\ra_\gamma$, we denote by $\Phi_\bfi(V)$ the flag variety of the word $\bfi$ for $V$. Namely, it consists of flags of the following form.
\begin{equation}\label{eq_flagF}
    \cF:\,  0=V_0\subset V_1\subset V_2\subset\cdots\subset V_n=V
\end{equation}
such that $\undd V_{k}/V_{k-1}=i_k$. 

Let $z\in \Lambda_\gamma$ and $\cF$ as in (\ref{eq_flagF}). We say $\cF$ \emph{stable} for $z$ if $z(V_p)\subset V_p$ for each $p\in [1,n]$. One sets the variety $\Lambda_{\bfi,\gamma}$ by
\begin{equation}\label{eq_lambfi}
    \Lambda_{\bfi,\gamma}:\Seteq{(\cF,z)\in \Phi_\bfi(V)\times \Lambda_\gamma}{z(V_k)\subset V_k \text{ for all } k\in[1,n]}
\end{equation}
There is a canonical map 
\begin{equation}\label{eq_grassman}
    \pi: \, \Lambda_{\bfi,\gamma}\to\Lambda_\gamma
\end{equation}
by sending $(\cF,z)$ to $z$. Let $\Gr_\bfi(z)$ denote the variety $\pi^{-1}(z)$.\\

\begin{definition}
 Let $Z$ be an irreducible component of $\Lambda_\alpha$. For a generic point $z\in Z$ and a word $\bfi\in\la Q_0\ra_\alpha$, we define the \emph{Serre polynomial} of $\Gr_\bfi(z)$ by
 \[\chi_q(\Gr_\bfi(z))=\sum_{k}(-1)^k\odim \HH_k(\Gr_\bfi(z))q^k\]
\end{definition}

\subsection{Induction functors}
Consider a partition $\apb=\gamma$ of a dimension vector $\gamma$ and a decomposition of $V=T\oplus W$ with $\undd V=\gamma$, $\undd T=\alpha$ and $\undd W=\beta$. Let $\Lambda_{\adb}$ be the set of pairs $(W',z)\in \Gr(\beta,\gamma)\times \Lambda_\gamma$ where $z(W')\subset W'$. Define $\Lambda_{\adb}^{(1)}$ as the variety of pairs $(W',z,g_1,g_2)$ where $z(W')\subset W'$, $g_1:W'\cong W$ and $g_2: V/W'\cong T$. Referring to \cite[diagram 1.10]{schiffmann2009lectures}, we have the diagram:
\begin{equation}
    \xymatrix{
    &\Lambda_\alpha\times\Lambda_\beta  &\Lambda_{\adb}^{(1)}\ar[l]_{p} \ar[r]^{r}&\Lambda_{\adb} \ar[r]^{q} &\Lambda_\gamma
    }
\end{equation}
where $r$ is $G_\alpha\times G_\beta$-torsor.  The pull-back functor $r^*$ establishes an equivalence between $D_{G_\gamma}(\Lambda_\adb)$ and $D_{G_\gamma\times G_\alpha\times G_\beta}^b(\Lambda_{\adb}^{(1)})$. Let $r_b$ be the inverse of $r^*$. It is evident that $p$ is not generally smooth, which implies that $\odim p$ isn't well-defined. \\

\noindent
Define $\Ind_{\adb}(-,-)$ as follows:
\begin{equation}
    \begin{split}
        \Ind_{\adb}(-,-): \,D_{G_\alpha}^b(\Lambda_\alpha)\times  D_{G_\beta}^b(\Lambda_\beta)\to  D_{G_\gamma}^b(\Lambda_\gamma)\\
        (\cF,\cE)\mapsto q_!r_b p^{*}(\cF\boxtimes\cE)[-\odim \Lambda_{\adb}]
    \end{split}
\end{equation}
\begin{remark}
It is worth noting that we use $-\odim \Lambda_{\adb}$ in our definition instead of $\odim p$ as in \cite[Section 1.3]{schiffmann2009lectures}.  There are two reasons: First, we use the Borel-Moore homology, whose degree is the opposite degree of cohomology of $p_!\CC_X$ for a given algebraic variety $X$. On the other hand, the map $p$ is not smooth and its dimension can't be defined. 
\end{remark}

\noindent
Following the proof of \cite[Proposition 1.9]{schiffmann2009lectures}, we see that this product is associative. This product gives a multiplication structure on $K(\Lambda_Q):=\bigoplus_{\gamma\in\NN[I]}K_{G_\gamma}(\Lambda_\gamma)$.

\begin{corollary}\label{cor_bfI}
For a sequence of dimensional vectors $(V_k)_{k\in[1,n]}$ such that $\odim V_k=i_k$ for some integer $n$ and $\sum_ki_k=\alpha$, set $\bfI_k=\bfI_{\Lambda_{i_k}}$ as the constant sheave on $\Lambda_{i_k}$. We have
\[(\bfI_1\star\bfI_2\star\cdots\star \bfI_n)_x=q^{-d_\bfi}\HH^\bullet(\Gr_{\bfi}(x))=q^{-d_\bfi}(\pi_!\bfI)_x\]
where $\pi:\, \Lambda_{\bfi,\alpha}\to \Lambda_{\alpha}$, $d_\bfi=\odim \Lambda_{\bfi,\alpha}$, and $x\in \Lambda_{\alpha}$. 
\end{corollary}
\begin{proof}
    It is easy to see from the proof of \cite[Proposition 1.9]{schiffmann2009lectures}. 
\end{proof}

\begin{example}
    Let $V=pi$ for some vertex $i$ and integer $p>0$. For the vertex $i$,  Let $\bfI_i$ represent the constant sheaf on $\Lambda_i$ for the vertex $i$. It can be observed that $\bfI^{*p}=\HH_{\bullet}(\cB)=[p]!\bfI_{pi}$. For any two vertices $i$ and $j$ in the set $I$, if there is no directed edge connecting them, then $\bfI_i*\bfI_j=\bfI_j*\bfI_i=\bfI_{i+j}$. In the case where there exists an arrow from vertex $i$ to vertex $j$ in the graph $Q$,  it follows
    \[\mu_{2i+j}^{-1}(0)=\Hom_k(\CC_i^2,\CC_j)\times \{0\}\cup \{0\}\times \Hom_k(\CC_j,\CC_i^2)=Z_1\cup Z_2 \]
    We can observe that $\odim\Lambda_{i,j,i}=\odim \mu_{2i+j}^{-1}(0)=2$ and $\odim\Lambda_{i,i,j}=\odim\Lambda_{j,i,i}=\mu_{2i+j}^{-1}(0)+1=3.$ 
   It implies 
    \[\bfI_i\star\bfI_j\star\bfI_i(x)=
    \begin{cases}
        q^{-2}\chi_q(pt)=q^{-2}    \text{ if } x\in (Z_1\cup Z_2)\setminus \{0\}\\
        q^{-2}\chi_q(\PP^1)=q^{-2}(1+q^2)  \text{ if } x=0,
    \end{cases}\]
      \[\bfI_i\star\bfI_i\star\bfI_j(x)=
    \begin{cases}
        q^{-3}\chi_q(\PP^1)=q^{-3}(1+q^2)   \text{ if } x\in Z_1\\
        q^{-3}\chi_q(\emptyset)=0 \text{ if } x\in Z_2\setminus\{0\},
    \end{cases}\]
    and 
     \[\bfI_j\star\bfI_i\star\bfI_i(x)=
    \begin{cases}
        q^{-3}\chi_q(\emptyset)=0    \text{ if } x\in Z_1\setminus\{0\}\\
       q^{-3} \chi_q(\PP^1)=q^{-3}(1+q^2) \text{ if } x\in Z_2.
    \end{cases}\]
    Therefore, we obtain
    \begin{equation}\label{eq_serrerelation}
        (q^{-1}+q)\bfI_i\star \bfI_j\star \bfI_i=\bfI_i\star\bfI_i\star\bfI_j+\bfI_j\star\bfI_i\star\bfI_i+g(q)\bfI_0, 
    \end{equation}
    where $g(q)=q^{-1}(q^{-1}+q)^2-2(q^{-3}+q^{-1})$. 
\end{example}
The equation mentioned above bears a resemblance to the Serre relation. To eliminate the $I_0$, we define the concept of being \emph{almost equal} for two sheaves $\mathcal{F}$ and $\mathcal{E}$ in $D_{G}^b(\Lambda_\gamma)$.
\begin{definition}\label{def_almost}
    For two sheave $\cF,\cE$ in $D_{G}^b(\Lambda_\gamma)$, $\mathcal{F}$ and $\mathcal{E}$ are considered \emph{almost equal} if they coincide on a dense open subset of $\Lambda_\gamma$. Specifically, there exists an open dense subset $U$ for each irreducible component $Z$ of $\Lambda_\gamma$ such that $\mathcal{F}_{\mid U} = \mathcal{E}_{\mid U}$. The relation denoted by $\sim$ captures this notion of almost equality. The quotient space $\tM{K}(\Lambda_Q)$ is then formed by taking the quotient of $K(\Lambda_Q)$ by the almost equal relation.

    It is noteworthy that for any dimension vector $\alpha$ and any sheaf $\mathcal{F} \in \tM{K}_{G_\alpha}(\Lambda_\alpha)$, we can express $\mathcal{F}$ as $\mathcal{F} \sim \sum_i \mathcal{F}_{z_i} \bfI_{Z_i}$ where $z_i$ lies in the open subset $U_i$ of $Z_i$, and $\mathcal{F}_{z_1} \cong \mathcal{F}_{z_2}$ for any $z_1, z_2 \in U_i$.
\end{definition}

\begin{remark}
Furthermore, it is important to note that the induction functor $m$ acting on $K(\Lambda_Q)$ does not induce a product structure on $\mathcal{K}(\Lambda_Q)$. For instance, in the case of $A_2: 1 \to 2$, where $\alpha = \alpha_1 + \alpha_2$ and $\beta = \alpha_2$, and considering $I_0 \sim 0$ as the constant sheaf with support $\{0\} \subset \Lambda_\alpha$, we observe that $\bfI_\beta \star \bfI_0 \nsim 0$. This is due to the fact that an extension of $S_1 \oplus S_2$ and $S_2$ results in $M[1,2] \oplus S_2$, where $M[1,2]$ represents a generic point of $\Lambda_\alpha$. Consequently, $M[1,2] \oplus S_2$ is a generic point of $\Lambda_{\alpha+\beta}$, leading to $\bfI_\beta \star \bfI_0 \neq 0$. 
\end{remark}

We see that the equation (\ref{eq_serrerelation}) implies that
\begin{equation}\label{eq_serre}
     (q^{-1}+q)\bfI_i\star \bfI_j\star \bfI_i\sim\bfI_i\star\bfI_i\star\bfI_j+\bfI_j\star\bfI_i\star\bfI_i
\end{equation}
It is the Serre relation.

\subsection{Multiplication on constructible sheaves}
For a dimension vector $\gamma\in\NN[Q_0]$ and a partition $\apb=\gamma$, we define, as in \cite{lusztig2000semicanonical}, a pairing 
\begin{equation}
    M_{G_\alpha}(\Lambda_\alpha)\times M_{G_\beta}(\Lambda_\beta)\to M_{G_\gamma}(\Lambda_\gamma)
\end{equation}
by $(f,f')\mapsto f\star f', (f\star f')(x)=\int_{q^{-1}(x)}\phi$, where $\phi:q^{-1}(x)\to \QQ$ is given by $\phi(W')=f(x_{\mid V/W'})f'(x_{\mid W'})$.

\begin{theorem}\label{theo_Psi}
    The map $\Psi:K_G(\Lambda)\to M_G(\Lambda)$ satisfies 
    \[\Psi(\cF)\star \Psi(\cE)=\Psi(\cF\star\cE)\]
\end{theorem}
\begin{proof}
   If $\cF$ and $\cE$ are constant sheave with supports $X$ and $Y$ respectively, we see that
    \begin{align*}
        \Psi(\cF\star \cE)(x)=&\sum_i(-1)^i\odim\HH_i(\cF\star\cE)_x\\
        =&\sum_i(-1)^i\odim\HH_i(q_!r_b p^*(\cF\boxtimes\cE))_x\\
        =& \sum_i(-1)^i\odim\HH_i(q^{-1}(x),r_b p^*(\cF\boxtimes\cE))\\
        =&\sum_i(-1)^{i}\odim \HH_i(q^{-1}(x)\cap rp^{-1}(X\times Y), \CC)\\
        =&\chi(q^{-1}(x)\cap rp^{-1}(X\times Y))\\
        =&\Psi(\cF)\star \Psi(\cE).
    \end{align*}
Because any constructible sheaves can be written as finite sums of constant sheaves and the maps $\star, \Psi$ commute with the sum in $K_G(\Lambda)$, we can see that our claim holds.  
\end{proof}

\section{Geometrization of semicanonical bases of quantum groups}
Let us recall the definition of quantum groups. The quantum group of symmetric type is an $\CC(q)$-algebra $U_q(\fg_Q)$ generated by $e_i,f_i$ for $i\in Q_0$ and $q^h$ for $h\in P^\vee$ subject to the following relations:
\begin{align*}
& q^0=1, q^h q^{h^{\prime}}=q^{h+h^{\prime}} \quad \text { for } h, h^{\prime} \in \mathrm{P}^{\vee}, \\
& q^h e_i q^{-h}=q^{\left\langle h, \alpha_i\right\rangle} e_i, \quad q^h f_i q^{-h}=q^{-\left\langle h, \alpha_i\right\rangle} f_i \quad \text { for } h \in \mathrm{P}^{\vee}, i \in I, \\
& e_i f_j-f_j e_i=\delta_{i j} \frac{t_i-t_i^{-1}}{q_i-q_i^{-1}}, \quad \text { where } t_i=q^{\mathrm{s}_i h_i}, \\
& \sum_{r=0}^{2}(-1)^r\left[\begin{array}{c}
2 \\
r
\end{array}\right]_i e_i^{2-r} e_j e_i^r=0 \quad \text { if } i \neq j, \\
& \sum_{r=0}^{2}(-1)^r\left[\begin{array}{c}
2 \\
r
\end{array}\right]_i f_i^{2-r} f_j f_i^r=0 \quad \text { if } i \neq j .
\end{align*}
Let $U_q(\fn)$ denote the subalgebra of $U_q(\fg_Q)$ generated by $e_i$. We will prove $\tM{K}^{nil}(\Lambda_Q)\cong U_q(\fn)$. To do this, we need a lemma as follows. 
\begin{lemma}\label{lem_words}
    For any word $\bfi=\squw{i}{n}$, we have
    \[\bfI_{i_1}\star\cdots\star\bfI_{i_n}\nsim 0. \]
\end{lemma}
\begin{proof}
   Let $\alpha=\sum_k i_k$. Suppose $\bfI_{i_1}\star\cdots\star\bfI_{i_n}\sim 0$, then there exists an open dense subvariety $U_i$ of each irreducible component $Z_i$ of $\Lambda_\alpha$ with $\bfI_{i_1}\star\cdots\star\bfI_{i_n}(U_i)=0$.  By applying the functor $\Psi$ 
 to $\bfI_{i_1}\star\cdots\star\bfI_{i_n}$, we obtain $\Psi(\bfI_{i_1}\star\cdots\star\bfI_{i_n})(U_i)=0$. Moreover, according to Theorem \ref{theo_Psi}, $\Psi(\bfI_{i_1}\star\cdots\star\bfI_{i_n})=e_{i_1}\cdots e_{i_n}\in U(\fn)$ as described in \cite{lusztig2000semicanonical}. Recall that the set of irreducible components of $\Lambda_\alpha$ forms a dual semicanonical base for $U(\fn)_\alpha^*$, where bilinear form is given by $(f,z)=f(z)$ for $f\in M_{G_\alpha}(\Lambda_\alpha)$ and $z$ in a dense open subset of an irreducible component of $\Lambda_\alpha$. Since this bilinear form is non-degenerated, there exists an open subset $W_k$ within an irreducible component $Z_k$ of $\Lambda_\alpha$ such that $\Psi(\bfI_{i_1}\star\cdots\star\bfI_{i_n})_{\mid W_k}$ is a nonzero constant in $\CC$. However, the intersection of $W_k$ and $U_k$ is not empty, leading to a contradiction. Thus, our assertion is proven. 
\end{proof}
\noindent
Even though it is not possible to precisely define the product on $\tM{K}(\Lambda_Q)$, it is feasible to define $\tM{\bfI}_\bfi\bar{\star}\tM{\bfI}_\bfj$ on $\tM{K}(\Lambda_Q)$ by 
\begin{equation}\label{eq_product}
\tM{\bfI}_\bfi\bar{\star}\tM{\bfI}_\bfj=\tM{\bfI}_{\bfi\bfj}
\end{equation}
for any pair of words $\bfi,\bfj$. Here, $\tM{\bfI}_\bfi$ represents the image of $\bfI_\bfi$ in $\tM{K}(\Lambda_Q)$. The subalgebra $\tM{K}^{nil}(\Lambda_Q)$ is defined as the algebra generated by $\tM{\bfI}_i$ for all $i\in Q_0$. Any element $x$ in $\tM{K}^{nil}(\Lambda_Q)$ can be expressed as $x=\sum_{\bfi}g_\bfi\bfI_\bfi$, where $g_\bfi\in\CC(q)$. Consequently, the product on $\tM{K}^{nil}(\Lambda_Q)$ is well defined. It is evident that this product on $\tM{K}^{nil}(\Lambda_Q)$ is associative. 
\begin{theorem}\label{theo_Phi}
    Let $Q$ be a quiver such that there exists at most one edge between any two vertices of $Q$, then there is a $\CC(q)$-algebra isomorphism 
    \[\Phi:\,U_q(\fn)\cong \tM{K}^{nil}(\Lambda_Q).\]
    by sending $e_i$ to $\bfI_i$. 
    Moreover, the set of the constant sheaves $\bfI_{Z}$ of the irreducible components $Z$ of $\Lambda_\gamma$ forms a base for $\tM{K}^{nil}(\Lambda_\gamma)$ for each dimension vector $\gamma\in\NN[Q_0]$. 
\end{theorem}
\begin{proof}
Consider the algebra $U'_q(\mathfrak{n})$ generated freely by elements $e_i$ for all $i\in Q_0$. An algebraic map $$\Psi:U'_q(\mathfrak{n})\to \mathcal{M}^{nil}(K)(\Lambda_Q)$$ is defined by mapping $e_i$ to $\tM{\bfI}_i$ for any $i\in Q_0$. The kernel of $\Psi$ is a two-sided ideal of $U'_q(\mathfrak{n})$, denoted by $I$. The Serre relation (\ref{eq_serre}) implies that 
\[(q+q^{-1})e_ie_je_i-(e^2_ie_j+e_je_i^2)\]
belongs in $I$ if there is an arrow $i\to j$ or $j\to i$. From this observation, we obtain a surjective map from $U_q(\fn)$ to $\tM{K}^{nil}(\Lambda_Q)$, which is denoted by $\Phi$.

 To prove that $\Phi$ is isomorphic, it needs to be shown that for any dimension vector $\alpha$ there exists a $\CC(q)$-module isomorphism $\tM{K}^{nil}(\Lambda_\alpha)\cong U_q(\fn_Q)_\alpha$. Because $U_q(\fn)_\alpha\otimes_{\CC(q)}\CC\cong U(\fn_Q)_\alpha$ and the set of irreducible components of $\Lambda_\alpha$ forms a base of $U(\fn_Q)_\alpha$ by \cite{lusztig2000semicanonical}, we see that the set of irreducible components of $\Lambda_\alpha$ forms a base of $U_q(\fn_Q)_\alpha$. 
    
    To prove $\tM{K}^{nil}(\Lambda_\alpha)\cong U_q(\fn_Q)_\alpha$,  it is necessary to show that the constant sheaf $\bfI_{Z_i}$ of each irreducible component $Z_i$ of $\Lambda_\alpha$ forms the basis of $\tM{K}^{nil}(\Lambda_\alpha)$. We have to show $\bfI_{Z_i}\in \tM{K}^{nil}(\Lambda_\alpha)$ for each irreducible component $Z_i$. Following the proof of \cite[Lemma 2.4]{lusztig2000semicanonical}, we can prove this by induction on the dimension of $V$. It is easy to see that this holds when $V=0$. We assume that this holds for any $I$-graded vector $\alpha'$ such that $\alpha'<\alpha$. 

    Let \[\Lambda_{V,i,p}:\Seteq{z\in \Lambda_\alpha}{\odim V_i/(\sum_{h; t(h)=i}\Ima z_{h})=p}\]
    and set $t_i(Z)=p$ if $Z\cap \Lambda_{V,i,p}$ is an open subset of $Z$.
    For our $\alpha$ we fix $i\in Q_0$ and we shall prove that 
    \begin{center}
       $\bfI_{Z}\in \tM{K}^{nil}(\Lambda_\alpha)$ for any $Z$ such that $t_i(Z)>0$. 
    \end{center} 
    As $t_i(Z)\leq \odim V_i$, we prove this by induction on $t_i(Z)$. We assume that $t_i(Z)=p>0$ and this holds for any $Z'$ such that $t_i(Z')=p'>p$. It is easy to see that $Z_1=Z\cap \Lambda_{V,i,p}$ is an open dense subvariety of $Z$. 
   
    Set $V''+V'=V$ with $\undd V'=pi$.
    Following \cite[Lemma 12.5]{lusztig1991quivers}, we see that there exists a unique irreducible component $Z'_1\in \Irr\Lambda_{V'',i,0}$ such that $x_{\mid V''}\in Z'_1$ for all $x\in Z_1$. By the first induction, the constant sheave $\bfI_{Z'_1}\in \tM{K}^{nil}(\Lambda_Q)$. Define $\tM{\cF}=\bfI_{V'}\star \bfI_{Z'_1}$. For any $x\in Z_1$, we have $\Gr_{V''}(x)=pt$. It implies that $\tM{\cF}_{\mid Z_1}=\bfI_{Z_1}$ and $\tM{\cF}_{\mid Z}=0$ for any $Z\in \Irr \Lambda_{V,i,p}\setminus\{Z_1\}$. Meanwhile, for any $x\in \Lambda_{V,i,<p}$, we have $\Gr_{V''}(x)=\emptyset$. It follows that $\tM{\cF}(x)=0$ for any $x\in \Lambda_{V,i,<p}$. For any irreducible component $Z'\in \Irr \Lambda_{V,i,>p}$, we have $\bfI_{Z'}\in \tM{K}^{nil}(\Lambda_V)$, then we define 
    \[\cF=\tM{\cF}-\sum_{Z'\in \Irr \Lambda_{V,i,>p}}\tM{\cF}_{Z'}\bfI_{Z'}\]
    We see that $\cF_{\mid Z}\sim\bfI_{Z}$ and $\cF_{\mid Z'}\sim0$ for all $Z'\in \Irr \Lambda_{V}\setminus\{Z\}$. It implies that $\cF\sim\bfI_{Z}\in \tM{K}^{nil}(\Lambda_V)$. 

    Therefore, we have $\odim \tM{K}^{nil}(\Lambda_\alpha)=\odim U_q(\fn)_\alpha$ for any $\alpha\in\NN[Q_0]$. It follows that $\Phi$ is injective and then bijective. 
 \end{proof}

\section{Dual semicanonical bases of quantum unipotent groups}
In this section, we will define the semicanonical base of quantum unipotent groups. It is well known that the graded dual space of $U_q(\fn)$ is quantum unipotent group $A_q(\fn)$. It is easy to see that the set of irreducible components $Z_i$ of $\Lambda_\gamma$ forms a base of $A_q(\fn)_\gamma$ for any dimension vector $\gamma\in \NN[I]$ via the following non-degenerated bilinear map
\begin{equation}\label{eq_bilinear}
    \begin{split}
        \la-,-\ra:\, \tM{K}^{nil}(\Lambda_\gamma)&\times\CC(q)\la\Irr \Lambda_\gamma\ra \to \CC(q)\\ 
        (\cF&, Z)\mapsto \sum_i\odim (-1)^i\HH_i(\cF_{z})q^i
    \end{split}
\end{equation}
where $z$ is a generic point of $Z$. Let $\tM{K}^*(\Lambda_\alpha)$ denote the dual space of $\tM{K}^{nil}(\Lambda_\alpha)$. By the above discussions, we have a $\CC(q)$-module isomorphism
\begin{equation}\label{eq_isoAq}
  \Phi^*:\,  A_q(\fn_Q)\cong \tM{K}^*(\Lambda_Q)
\end{equation}
We define the product on $\tM{K}^*(\Lambda_Q)$ by 
\[\bar{m}:=\Phi^* m(\Phi^{*-1}\times \Phi^{*-1}):\,\tM{K}^*(\Lambda_Q)\times\tM{K}^*(\Lambda_Q)\to \tM{K}^*(\Lambda_Q).\]
It follows 
\begin{equation}\label{eq_barm}
\bar{m}\circ \Phi^*=\Phi^*\circ m.     
\end{equation}

\subsection{Shuffle algebras}
In this section, we will embed the algebra $\tM{K}^*(\Lambda_Q)$ into the shuffle algebras. Let $\CC(q)\la Q_0\ra_\gamma$ be the $\CC(q)$-module freely generated by all words $\bfi\in\la Q_0\ra_\gamma$. Now we will define the multiplication for two words $\bfi=[i_1\cdots i_m]$ and $\bfj=[i_{m+1},\cdots i_{m+n}]$:
\[\bfi\circ\bfj=\sum_{\sigma} q^{-e(\sigma)}[i_{\sigma(1)}\cdots i_{\sigma(m+n)}]\]
where the sum runs over $\sigma\in S_{m+n}$ such that $\sigma(1)<\cdots<\sigma(m)$ and $\sigma(m+1)<\cdots\sigma(m+n)$, and $$e(\sigma)=\sum_{k \leqslant m<l ; \sigma(k)<\sigma(l)}\left(\alpha_{i_{\sigma(k)}}, \alpha_{i_{\sigma(l)}}\right) .$$
Thus, we obtain a shuffle product on $\CC(q)\la Q_0\ra:=\bigoplus_{\alpha\in \NN[Q_0]}\CC(q)\la Q_0\ra_\alpha$. Following \cite{leclerc2004dual}, one has an algebra embedding:
\begin{equation}\label{eq_Aqshuffle}
 \varPhi :\,  A_q(\fn_Q)\hookrightarrow \CC(q)\la Q_0\ra
\end{equation}

For a dimension vector $\alpha$, we define the $\CC(q)$-module map $\Psi$ by 
\begin{equation}
    \begin{split}
        \varPsi:\, \tM{K}^*(\Lambda_\alpha)&\to \CC(q)\la Q_0\ra_\alpha\\
                 Z&\mapsto \sum_{\bfi\in\la I\ra_\alpha} q^{-d_\bfi}\chi_q(\Gr_\bfi(z))\bfi 
    \end{split}
\end{equation}
where $z$ is a generic point of $Z$, and $\chi_q$ denotes the Serre polynomial of variety $\Gr_\bfi(z)$. Lemma \ref{lem_words} implies that the map $\varPsi$ is injective. The equations (\ref{eq_Aqshuffle}) and (\ref{eq_barm}) implies that $\Psi$ is an algebraic map. Hence, we obtain the following theorem.

\begin{theorem}
    For a quiver $Q$ without multiple arrows between any two vertices, there exists a base of $A_q(\fn_Q)$ indexing by the irreducible components of Lusztig's nilpotent varieties, and they can be written as the following form (up to some $q$-powers):
    \[\delta_Z:=\sum_{\bfi\in\la Q_0\ra_\alpha} q^{-d_\bfi}\chi_q(\Gr_\bfi(z))\bfi \]
     where $Z$ is an irreducible component of $\Lambda_\alpha$, $z$ is a generic point of $Z$, and $d_\bfi=\odim\Lambda_{\bfi,\alpha}$.
\end{theorem}
\begin{proof}
    Let $\circ$ be the product of $k\la Q_0\ra$. Following \cite{leclerc2004dual}, we see that $\varPhi$ is given by \[\varPhi(x)=\sum_\bfi(e_\bfi,x)\bfi. \]
    By bilinear form (\ref{eq_bilinear}), Corollary \ref{cor_bfI}, and Theorem \ref{theo_Phi}, we get 
     \[\delta_Z:=\sum_{\bfi\in\la Q_0\ra}(\bfI_\bfi,Z)\bfi=\sum_{\bfi\in\la Q_0\ra_\alpha} q^{-d_\bfi}\chi_q(\Gr_\bfi(z))\bfi \]
    where $Z$ is an irreducible component of $\Lambda_\alpha$, $z$ is a generic point of $Z$, and $d_\bfi=\odim\Lambda_{\bfi,\alpha}$. 
\end{proof}

\end{document}